\documentclass[a4paper, 11pt]{amsart}

\usepackage[T1]{fontenc}
\usepackage{lmodern}
\usepackage{amsmath, amssymb, amsthm, bm, cancel}
\usepackage[margin=2.5cm]{geometry}
\usepackage[dvipsnames]{xcolor}

\usepackage[%
  bookmarks=true,		
  unicode=true,			
  pdftitle={Poisson structures on the conifold and local Calabi--Yau threefolds},	%
  pdfauthor={Elizabeth Gasparim}{Bruno Suzuki},	%
  pdfkeywords={}{}{}{},	
  colorlinks=true,		
  linkcolor=Blue,		
  citecolor=cite,		
  filecolor=magenta,	
  urlcolor=RoyalBlue	
]{hyperref}				

\usepackage{cleveref}
\usepackage{epsf,amsfonts}
\usepackage{enumerate}
\usepackage{setspace, ifthen}

\usepackage{tikz-cd}
\usetikzlibrary{patterns, calc, arrows}
\colorlet{cite}{LimeGreen!50!Green}

\newcommand{\ce}{\mathrel{\mathop:}=}
\newcommand{\pd}[1]{\frac{\partial}{\partial #1}}
\newcommand{\trwdg}[3]{\partial_{#1} \wedge \partial_{#2} \wedge \partial_{#3}}
\newcommand{\pbra}[4]{#1^{#2} #1^{#3}_{,#4} - #1^{#3} #1^{#2}_{,#4}}

\numberwithin{equation}{section}

\theoremstyle{definition}
\newtheorem{note}[equation]{Note}
\newtheorem{definition}[equation]{Definition}
\newtheorem{theorem}[equation]{Theorem}
\newtheorem*{theorem*}{Theorem}
\newtheorem{proposition}[equation]{Proposition}
\newtheorem{lemma}[equation]{Lemma}
\newtheorem*{lemma*}{Lemma}
\newtheorem{corollary}[equation]{Corollary}
\newtheorem*{corollary*}{Corollary}
\newtheorem{notation}[equation]{Notation}
\newtheorem{remark}[equation]{Remark}

\DeclareMathOperator{\Tot}{Tot}
\renewcommand{\H}{\operatorname{H}}

\newcommand{\SC}{0.7}
\newcommand{\Zmone}{
\begin{minipage}{0.2\textwidth} 
\centering \vspace{4pt}
\begin{tikzcd}
	\fill[pattern=north west lines, pattern color=SkyBlue] (0,0) -- (0,1*\SC) -- (-1*\SC,1*\SC) -- (-1*\SC,-1*\SC);
	\fill[pattern=north east lines, pattern color=SkyBlue] (0,0) rectangle (1*\SC,1*\SC); 
	\draw[very thick,->] (0,0) -- (1*\SC,0);
	\draw[very thick,->] (0,0) -- (0,1*\SC);
	\draw[very thick,->] (0,0) -- (-1*\SC,-1*\SC);
	\begin{scope}
		\foreach \x in {-1,...,1}
		{	\foreach \y in {-1,0,1}
			{
			\filldraw (\x*\SC,\y*\SC) circle (0.03cm);
			}
		}
	\end{scope}	
\end{tikzcd}
\vspace{4pt}
\end{minipage}
}
\newcommand{\Zthree}{
\begin{minipage}{0.2\textwidth}
\centering \vspace{4pt}
\begin{tikzcd}
	\fill[pattern=north west lines, pattern color=SkyBlue] (0,0) -- (-1*\SC,3*\SC) -- (0,3*\SC);
	\fill[pattern=north east lines, pattern color=SkyBlue] (0,0) rectangle (1*\SC,3*\SC); 
	\draw[very thick,->] (0,0) -- (1*\SC,0);
	\draw[very thick,->] (0,0) -- (0,1*\SC);
	\draw[very thick,->] (0,0) -- (-1*\SC,3*\SC);
	\begin{scope}
		\foreach \x in {-1,...,1}
		{	\foreach \y in {0,...,3}
			{
			\filldraw (\x*\SC,\y*\SC) circle (0.03cm);
			}
		}
	\end{scope}	
\end{tikzcd}
\vspace{4pt}
\end{minipage}
}	

\newcommand{\ccfan}{
\begin{minipage}{0.1\textwidth}
\centering \vspace{4pt}
\begin{tikzcd}
	\fill[pattern=north west lines, pattern color=SkyBlue] (0,0) rectangle (1*\SC,1*\SC);

	\draw[very thick,->] (0,0) -- (1*\SC,0);
	\draw[very thick,->] (0,0) -- (0,1*\SC);
	\begin{scope}
		\foreach \x in {0,1}
		{	\foreach \y in {0,1}
			{
			\filldraw (\x*\SC,\y*\SC) circle (0.03cm);
			}
		}
	\end{scope}	
\end{tikzcd}
\vspace{4pt}
\end{minipage}
}
\newcommand{\cfan}{
\begin{minipage}{0.05\textwidth}
\centering \vspace{4pt}
\begin{tikzcd}
	\draw[very thick,->] (0,0) -- (0,1*\SC);
	\begin{scope}
		\foreach \x in {0}
		{	\foreach \y in {0,1}
			{
			\filldraw (\x*\SC,\y*\SC) circle (0.03cm);
			}
		}
	\end{scope}	
\end{tikzcd}
\vspace{4pt}
\end{minipage}
}

\newcommand{\chfan}{
\begin{minipage}{0.1\textwidth}
\centering \vspace{4pt}
\begin{tikzcd}
	\draw[very thick,->] (0,0) -- (1*\SC,0);
	\begin{scope}
		\foreach \x in {0,1}
		{	\foreach \y in {0}
			{
			\filldraw (\x*\SC,\y*\SC) circle (0.03cm);
			}
		}
	\end{scope}	
\end{tikzcd}
\vspace{4pt}
\end{minipage}
}

\newcommand{\Zzero}{
\begin{minipage}{.2\textwidth}
\centering \vspace{4pt}
\begin{tikzcd}
	\fill[pattern=north west lines, pattern color=SkyBlue] (0,0) rectangle (-1*\SC,1*\SC);
	\fill[pattern=north east lines, pattern color=SkyBlue] (0,0) rectangle (1*\SC,1*\SC); 
	\draw[very thick,->] (0,0) -- (1*\SC,0);
	\draw[very thick,->] (0,0) -- (0,1*\SC);
	\draw[very thick,->] (0,0) -- (-1*\SC,0);
	\begin{scope}
		\foreach \x in {-1,...,1}
		{	\foreach \y in {0,1}
			{
			\filldraw (\x*\SC,\y*\SC) circle (0.03cm);
			}
		}
	\end{scope}	
\end{tikzcd}
\vspace{4pt}
\end{minipage}
}

\author[Ballico Gasparim K\"{o}ppe Suzuki]{E. Ballico, E. Gasparim, T. K\"{o}ppe, B. Suzuki}
\address{EG, BS - Depto. Matem\'aticas, Univ. Cat\'olica del Norte, Antofagasta, 
Chile, etgasparim@gmail.com, obrunosuzuki@gmail.com. EB -  Dept. Mathematics, Univ. Trento,  Italy,
ballico@science.unitn.it.}
\title[Poisson structures on local Calabi--Yau threefolds]{Poisson structures on the conifold \\ and local Calabi--Yau threefolds}

\begin{document}

\begin{abstract}
We describe  bivector fields  and Poisson structures on   local Calabi--Yau threefolds
which are total spaces of vector bundles on a contractible rational curve.
In particular, we calculate all possible holomorphic Poisson structures on the  conifold.
\end{abstract}

\maketitle

\section{Motivation and results}

We are interested in  holomorphic Poisson structures on Calabi--Yau threefolds
that contain a contractible rational curve. Here we consider the local situation. Hence, 
we study  Poisson structures on Calabi--Yau threefolds that are the total space of 
a rank 2 vector bundle on  $\mathbb P^1$. A result of Jim\'enez \cite{J} 
says that the contraction of a smooth rational curve on a threefold may happen in exactly 3 cases, 
namely when the normal bundle to such a
curve is one of  
$$\mathcal O_{\mathbb P^1}(-1) \oplus \mathcal O_{\mathbb P^1}(-1), \quad 
\mathcal O_{\mathbb P^1}(-2) \oplus \mathcal O_{\mathbb P^1}(0), \quad \text{or} \quad
\mathcal O_{\mathbb P^1}(-3) \oplus \mathcal O_{\mathbb P^1}(1),$$
although only in the first case it can contract to an isolated singularity.
In this work we describe completely the local case, that is, 
we classify all isomorphism classes of holomorphic Poisson structures on  the local Calabi--Yau threefolds 
$$W_k\ce \Tot (\mathcal O_{\mathbb P^1}(-k) \oplus \mathcal O_{\mathbb P^1}(k-2)), \quad k=1,2,3,$$
calculate their  Poisson cohomology, describe their  symplectic foliations and some properties of their moduli. 
 
Polishchuk  shows a correspondence between Poisson structures on a scheme $X$ and a blow-up $\widetilde{X}$,
which applies to the  cases we study  \cite[Thm.\thinspace 8.2, 8.4]{Po}. Hence, 
describing Poisson structures on $W_k$  
is equivalent to describing Poisson structures on the singular threefolds obtained from them by contracting the rational curve
to a point.
 
Poisson structures are parametrized by those elements of  $H^0(W_k,\Lambda^2 TW_k)$ which are integrable. 
We briefly recall some basic definitions from Poisson cohomology, for details see \cite[Ch.\thinspace 4]{LPV}.
Let $(M, \pi)$ be a Poisson Manifold. The graded algebra $\mathfrak{X}^{\bullet}(M) = \Gamma(\wedge^\bullet TM)$ 
and the degree-$1$ differential operator $d_\pi = [\pi, \cdot]$ define the {\it Poisson Cohomology} of $(M, \pi)$. 
The first cohomology groups have clear geometric meaning:

$\H^0(M, \pi) = \ker [\pi, \cdot] = \mbox{Cas}(\pi) = $ holomorphic functions on $M$ which are constant along symplectic leaves.
These are the Casimir functions of $(M,\pi)$.

$\H^1(M, \pi) = \dfrac{\mbox{Poiss}(\pi)}{\mbox{Ham}(\pi)}$ is the quotient of Poisson vector fields by Hamiltonian vector fields.

We compute Poisson cohomology groups and use them to distinguish Poisson structures,  identifying their degeneracy loci. 

The \emph{$r^\text{th}$ degeneracy locus} of a holomorphic Poisson structure $\sigma$ on a complex manifold or algebraic variety $X$ is defined as
\[
D_{2r} (\sigma) \ce \{ x \in X \mid \mbox{rank}\, \sigma (x) \leq 2r \} \text{ ,}
\]
where $\sigma$ is viewed as a map $\mathcal T_X^* \to \mathcal T_X$ 
by contracting a $1$-form with the bivector field $\sigma$. 
At a given point on a complex { threefold} a holomorphic Poisson structure has either  rank $2$, 
or rank $0$. 
Therefore, for the threefolds $W_k$ we name 
$$D (\sigma) \ce D_0 (\sigma)$$
the \emph{degeneracy locus} of $\sigma$, hence it consists of points where $\sigma$ has rank 0.

A nondegenerate holomorphic Poisson structure $\sigma$ is called a {\it holomorphic symplectic} structure, 
given that $\sigma$ determines a nondegenerate closed holomorphic $2$-form $\omega$ by setting
\[
\omega (X_f, X_g) = \{ f, g \}_\sigma \text{ ,}
\]
where $X_f$ denotes the Hamiltonian vector field associated to a function $f$.

\begin{remark}\label{leaves}
Each Poisson structure determines a symplectic foliation, whose leaves consist of 
maximal symplectic submanifolds. 
In particular, in the case of threefolds, 
the degeneracy locus $D(\sigma)$ is formed by leaves  consisting  of a single point each,
and all other leaves have  complex dimension 2.
\end{remark}

Describing holomorphic Poisson structures on the Calabi--Yau threefolds $W_k$
can also be regarded as describing their first-order 
noncommutative deformations. The  commutative deformation theory of these threefolds $W_k$ and 
the structure of moduli of vector bundles on them is described in detail in \cite{BGS}.
The surfaces $Z_k \ce \Tot (\mathcal O_{\mathbb P^1}(-k))$,
which were discussed in \cite{BG1,BG2} and \cite{BeG}, occur here as useful tools. 

Motivated by the definition of {\it moduli space of Poisson structures} given in \cite[Sec.\thinspace 1.2]{Pym}, namely 
the quotient $\mbox{Poiss}(X)/\mbox{Aut}(X)$,
we also  describe some isomorphisms among Poisson structures. 
We note that the space of Poisson structures on a threefold can be seen as a cone over global functions in
the following sense:

\begin{proposition}\label{subm}
Let $X$ be a smooth complex threefold and $u$ a Poisson structure on $X$, i.e. an integrable bivector field. Then $fu$ is
integrable for all $f\in \mathcal O(X)$.
\end{proposition}

\begin{proof}
Since $fu$ is a holomorphic bivector, we only need to check that $fu$ is integrable. This is a local condition. Locally $u$ is
the product of an element of $K_X^{-1}$ and a holomorphic one-form $w$, and the integrability of $u$ is equivalent to
$w \wedge dw = 0$ \cite[Eq.\thinspace 4]{Pym}. We have $(fw)\wedge d(fw) = f^2w\wedge dw + w\wedge df\wedge w = 0 + 0$.
\end{proof}

Among our local threefolds,  the most famous is certainly $W_1$, 
known in the physics literature as the resolved conifold, since it occurs as the crepant resolution of 
the double point singularity $xy-zw=0$ in $\mathbb C^4$ known as the conifold. 
The conifold singularity is extremely popular in string theory 
because it can be resolved in two different ways, by
a 2-sphere (resolution) or a 3-sphere (deformation). This leads to what is known as a
geometric transition and establishes dualities between distinct theories in physics, such as 
gauge--gravity and open--closed string duality, see \cite{BBR} and references therein.
We start with  bivector fields on $W_1$:

\begin{lemma*}[\ref{biW1}]
The space  $M_1 = H^0(W_1,\Lambda^2TW_1)$
parametrizing all holomorphic bivector fields on $W_1$ 
has the following structure as a module over global holomorphic functions:
 $$M_1 = \langle e_1,e_2,e_3,e_4 \rangle / \langle zu_2e_1-zu_1 e_2- u_2 e_3 +u_1e_4 \rangle \text{ .}$$
\end{lemma*}

Describing obstructions to integrability, we obtain an explicit description of Poisson structures
on $W_1$. For ${\bf p}= (p^1,p^2,p^3,p^4) \in M_1$ we 
describe a differential operator
$B({\bf p}) = {\bf p}^tQ{\bf p}$ (\Cref{operator}).

\begin{theorem*}[\ref{t1}]
Every holomorphic Poisson structure on $W_1$ has the form $\sum_{i=1}^4p^ie_i$ where 
$ (p^1,p^2,p^3,p^4) \in  B^{-1}(0)$.
\end{theorem*}

\begin{theorem*}[\ref{iso}]
The Poisson structures $e_1,e_2,e_3,e_4$ are all  pairwise isomorphic.
\end{theorem*}

Since the generators give isomorphic Poisson structures, it is enough to describe the foliation 
corresponding to one of them, we choose $e_2= \partial_0 \wedge \partial_1$.

\begin{theorem*}[\ref{w1fol}]
The symplectic foliation for $(W_1, \partial_0 \wedge \partial_1)$ is given by:

\begin{itemize}
\item $ \partial_0 \wedge \partial_1$ has degeneracy locus on the line 
$\{v_2= \xi=0\}$, where the leaves are $0$-dimensional, consisting of single points, and 
\item
$2$-dimensional symplectic leaves cut out on the $U$ chart by $u_2$ constant.
\end{itemize}
\end{theorem*}
  
We summarize this result in a small table.

\begin{center}
\begin{tabular}{c|c|l}
\multicolumn{3}{c}{\sc $W_1$ Poisson structures }\\
\multicolumn{3}{c}{} \\
$\pi$ & degeneracy & Casimir \\ \hline
$e_2$  & \chfan &$ f(u_2)$
\end{tabular}
\end{center}
We then see the Poisson structures as determined by surface embeddings.
 
\begin{theorem*}[\ref{princ}] 
The 4 principal  embeddings of the Poisson surface $(Z_1,\pi_0)$ generate all Poisson structures on $W_1$.
\end{theorem*}

From the viewpoint of Poisson structures, $W_2$ is the best of our local  Calabi--Yaus, since it 
is the only one that admits a nondenegerate Poisson structure, see \Cref{degeneracy2}, 
although this comes as no surprise since $W_2 \simeq T^*\mathbb P^1\times \mathbb C$ 
is a product of symplectic manifolds. 

\begin{lemma*}[\ref{W2gens}]
The space  $M_2 $
of holomorphic bivector fields  on $W_2$ 
has the following structure as a module over global holomorphic functions:
$$M_2 = \langle e_1,e_2,e_3,e_4, e_5 \rangle / \langle u_1e_3-zu_1e_1, 
 u_2e_5-zu_2e_3-2zu_2e_2 \rangle \text{ .} $$
\end{lemma*}

By \Cref{iso1-5}, $e_1$ and $e_5$ give isomorphic Poisson structures, whereas 
the others are distinct, giving interesting symplectic foliations, as follows.

\begin{theorem*}[\ref{foliations2}]
The symplectic foliations on $W_2$ 
have $0$-dimensional leaves consisting of single points over each 
of their corresponding  degeneracy loci described in \Cref{degeneracy2},
and their generic leaves, which are $2$-dimensional, are as follows:
\begin{itemize}
\item surfaces of constant $u_1$  for $e_1$ and $e_3$,  one of them   isomorphic to $\mathbb P^1 \times \mathbb C$.
\item isomorphic to $\mathbb C-\{0\} \times \mathbb C$ for $e_2$ (contained in the fibers of the projection to $\mathbb P^1$).
\item isomorphic to the surface $Z_2$ and cut out by $u_2=v_2$ constant for $e_4$.
\end{itemize}
\end{theorem*}

We depict their degeneracy loci and Casimir functions in the following table.

 \begin{center}
\begin{tabular}{c|c|l}
\multicolumn{3}{c}{\sc $W_2$ Poisson structures }\\
\multicolumn{3}{c}{} \\
$\pi$ & degeneracy & Casimir \\ \hline
$e_1$ & \ccfan & $ f(u_1)$\\ \hline
$e_2$  & \Zzero& $ f(z)$ \\ \hline
$e_3$ &  \ccfan $\cup$ \cfan & $f(u_1)$ \\ \hline
$e_4$  & $\emptyset$ & $f(u_2)$ 
\end{tabular}
\end{center}

\vspace{3mm}

We then continue onto the case of $W_3$, obtaining:

\begin{lemma*}[\ref{W3gens}] 
The space  $M_3 $
of holomorphic bivector fields on $W_3$ 
has the following structure as a module over global holomorphic functions:
$$M_3= \mathbb C \langle e_1, \dots, e_{13} \rangle / R$$
     with the set of relations $R$ is the ideal generated by   the expressions  
     $$\begin{array}{l}
u_1  e_2  - u_1u_2 e_1 \\
u_1  e_{10} -u_1 u_2 e_3 \\
u_1e_{13} - u_1u_2e_7 \\
\end{array}
\quad\quad \begin{array}{l}
zu_1e_{12} - u_1u_2 e_6 \\
zu_1e_{13} - u_1u_2e_8  \\
u_1e_{11} - zu_1 e_{10} \\
\end{array}
\quad \begin{array}{l}
u_1 e_4 - z u_1e_3   \\
 u_1  e_5 - zu_1e_4  \\
u_1e_8 - zu_1e_7  \\
\end{array} \quad 
\begin{array}{l}
u_1 e_6 - zu_1 e_5 - 3z^2u_1e_1 \\
u_1e_9 - zu_1e_8 + zu_1e_2 \\
u_1e_{12} - zu_1e_{11} - 3zu_1e_1 .\\
\end{array}$$
\end{lemma*}
 
Subsequently, we describe some features of the symplectic foliations
corresponding to the generating Poisson structures on $W_3$.

\begin{theorem*}[\ref{W3foliation}]
The symplectic foliations on $W_3$ 
have 0-dimensional  leaves consisting of single points over each 
of their corresponding  degeneracy loci described in 
the proofs of Lemmata \ref{alphas}, \ref{betas}, \ref{gammas}, and their generic  leaves, which  are $2$-dimensional, 
are as follows:

\begin{itemize}
\item  Isomorphic to $\mathbb C^*\times \mathbb C$ for $e_1$.
\item  Isomorphic to $\mathbb C^*\times \mathbb C^*$ for $e_2$.
\item  Surfaces of constant $u_1$, for $e_3, e_4,e_5, e_{10}, e_{11}$ and $e_{13}$. 
\item Surfaces of constant $u_2$,  for $e_7$ and $e_8$.
\end{itemize}
\end{theorem*}
 
Geometrically, these structures can be obtained from surface embeddings.

\begin{theorem*}[\ref{emb3}]
The embeddings of  Poisson surfaces $j_0(\mathbb C^2, \pi_0),$ $j_1(Z_3,\pi_i)$ with $i=0,1,2$ and $j_2(Z_{-1},\pi_i)$ 
with $i=0,1,2,3$
 generate 
all Poisson structures on $W_3$.
\end{theorem*}

We finish by showing that except for $e_6,e_9, e_{12}$, the Poisson structures $e_i$
are all pairwise non-isomorphic, which can be seen from their degeneracy loci.

\begin{center}
\begin{tabular}{c|c|c}
\multicolumn{3}{c}{\sc $W_3$ Poisson structures }\\
\multicolumn{3}{c}{} \\
$\pi$ & degeneracy & Casimir \\ \hline
$e_1$   & \Zmone $\cup$ \ccfan & $f(z)$ \\ \hline
$e_2$ & \Zmone $\cup$ \Zthree & $ f(z)$ \\ \hline
$e_3$  & \ccfan & $ f(u_1)$ \\ \hline
$e_4$  &\ccfan $\cup$ \ccfan & $f(u_1)$ \\ \hline
$e_5$  & \ccfan $\cup$ \cfan & $ f(u_1)$ \\ \hline
$e_7$ & \Zmone $\cup$ \ccfan & $ f(u_2)$ \\ \hline
$e_8$ & \Zmone $\cup$ \ccfan $\cup$ \cfan & $ f(u_2)$ \\ \hline
$e_{10}$ & \Zthree $\cup$ \ccfan & $ f(u_1)$ \\ \hline
$e_{11}$ & \Zmone $\cup$ \ccfan $\cup$ \cfan  & $f(u_1)$ \\ \hline
$e_{13}$   & \Zmone $\cup$ \Zthree  & $ f(u_2)$
\end{tabular}
\end{center}

\begin{remark}
\Cref{biW1} (resp.\ \ref{W2gens}, \ref{W3gens}) shows that
the space of holomorphic bivector fields on $W_1$ (resp.\ $W_2$, $W_3$) is generated
as a module over global holomorphic functions by $4$ (resp.\ $5$,  $13$) holomorphic bivectors.
Moreover, we find that $\mathbb{C}$-linear combinations of the basis vectors are integrable,
and by \Cref{subm} any multiples thereof by holomorphic functions are integrable, too. In the
case of $W_1$, \Cref{t1} describes the space of integrable bivectors as the kernel of
an explicit differential operator.
\end{remark}

However, we note that general combinations of the  basis vectors with global functions as coefficients 
may not be integrable. For example, on $W_1$ the expression $zu_2e_1+e_3$ gives a nonintegrable bivector field, 
despite the fact that both $e_1$ and $e_3$ are integrable.

\section{Vector fields on \texorpdfstring{$W_k$}{W\_k}}

\begin{definition}\label{WKdef}For integers $k_1$ and $k_2$, we set
\[
W_{k_1,k_2} = \Tot (\mathcal{O}_{\mathbb{P}^1}(-k_1) \oplus \mathcal{O}_{\mathbb{P}^1}(-k_2)) \text{ .}
\]
The complex manifold structure  can be described by gluing the open sets 
$$U = \mathbb{C}^3_{\{z,u_1,u_2\}}  \quad \mbox{and} \quad  V = \mathbb{C}^3_{\{\xi,v_1,v_2\}}$$ 
by the relation
\begin{equation}\label{canonical}
(\xi, v_1, v_2) = (z^{-1}, z^{k_1} u_1, z^{k_2} u_2)
\end{equation}
whenever $z$ and $\xi$ are not equal to 0.
We call \eqref{canonical} the canonical coordinates for $W_{k_1,k_2}$.
\end{definition}

\begin{lemma}
The threefold $W_{k_1,k_2}$ is Calabi--Yau if and only if $k_1 + k_2 = 2$.
\end{lemma} 
 
\begin{proof}
The canonical bundle is given by the transition 
$-z^{k_1+k_2-2}$, 
so it is trivial if and only if $k_1 + k_2 = 2$.
\end{proof} 

\begin{notation}
We denote by $W_k$ the Calabi--Yau threefold 
\[
W_k\ce W_{k,-k+2} = \Tot (\mathcal{O}_{\mathbb{P}^1}(-k) \oplus \mathcal{O}_{\mathbb{P}^1}(k+2)) \text{ .}
\]
\end{notation}

Let $U \subset W_k$ be our usual chart with coordinates $\{z, u_1, u_2\}$.
As a module over the ring of functions $H^0(U; \mathcal{O})$, the module of
global sections of vector fields over $U$, $H^0(U; \; TU)$ is spanned by the
coordinate partial derivatives, which we relabel for convenience:
\[ \pd{z}    \equiv \frac{\partial}{\partial x^0} \equiv \partial_0 \text{ ,\quad}
   \pd{u_1}  \equiv \frac{\partial}{\partial x^1} \equiv \partial_1 \text{ , and }
   \pd{u_2}  \equiv \frac{\partial}{\partial x^2} \equiv \partial_2 \text{ .}
\]
The exterior powers are spanned by the appropriate wedge products:
\begin{eqnarray*}
  H^0(U; \; \Lambda^1 TU) & = & \bigl\langle \left\{ \partial_i \right\}_{i=0}^{2} \bigr\rangle \\
  H^0(U; \; \Lambda^2 TU) & = & \left\langle \left\{ b_0 \equiv \partial_1 \wedge \partial_2 \text{ , \ }
                                                     b_1 \equiv \partial_2 \wedge \partial_0 \text{ , \ }
                                                     b_2 \equiv \partial_0 \wedge \partial_1 \right\} \right\rangle \\
  H^0(U; \; \Lambda^3 TU) & = & \left\langle \left\{ \partial_0 \wedge \partial_1 \wedge \partial_2 \right\} \right\rangle 
\end{eqnarray*}
We are interested in bivectors, i.e.\ elements of $H^0(U; \; \Lambda^2 TU)$.
We write a bivector field as
\begin{equation}\label{eq.defq}
  q = q^i b_i = \frac{1}{2} q^i \varepsilon_{i}^{jk} \partial_j \wedge \partial_k \text{ ,}
\end{equation}
where the coefficients $q^i$ are functions on $U$.
We are using Einstein summation convention throughout,
and we write $f_{,i}$ for $\frac{\partial f}{\partial x^i}$.


We collect a few useful identities involving Lie brackets, and some
preliminary expressions used to compute the Schouten--Nijenhuis
brackets. Let $X$, $Y$ be vector fields and $f$, $g$ be functions. Then:
\begin{itemize}
\item
  For the coordinate partial derivatives, the Lie bracket vanishes: $[\partial_j, \partial_k] = 0$ for all $j$, $k$.
\item
  $[X, g Y] = X(g)Y + g [X,Y]$, so in particular, $[\partial_j, g \partial_k] = \frac{\partial g}{\partial x^j}\partial_k$
  and $[f \partial_j, \partial_k] = -\frac{\partial f}{\partial x^k}\partial_j$.
\item
  The SN-bracket of two bivectors is commutative and results in a degree-$3$ trivector,
  i.e.\ a scalar multiple of $\partial_0 \wedge \partial_1 \wedge \partial_2$. On basis elements, it is given by:
  \begin{align}\label{eqn.snbasic}
    \bigl[ f \, \partial_j \wedge \partial_k, g \, \partial_m \wedge \partial_n \bigr] = & \ 
    \bigl[ f \, \partial_j, g \, \partial_m \bigr] \wedge \partial_k \wedge \partial_n - \nonumber
    \bigl[ f \, \partial_j, \partial_n \bigr] \wedge \partial_k \wedge g \, \partial_m \\ & -
    \bigl[ \partial_k, g \, \partial_m \bigr] \wedge f \, \partial j \wedge \partial_n +
    \cancel{\bigl[ \partial_k, \partial_n \bigr]} \wedge f \, \partial_j \wedge g \, \partial_m \nonumber \\ = & \ 
    f g_{,j} \,\trwdg{m}{k}{n} - g f_{,m} \,\trwdg{j}{k}{n} \nonumber \\ & +
    g f_{,n} \,\trwdg{j}{k}{m} - f g_{,k} \,\trwdg{m}{j}{n}.
  \end{align}
\end{itemize}
We are now in a position to compute the self-bracket of a general bivector field $q$:
\[ \bigl[ q, q \bigr] =
   \bigl[ q^0 \partial_1 \wedge \partial_2 + q^1 \partial_2 \wedge \partial_0 + q^2 \partial_0 \wedge \partial_1, \;
          q^0 \partial_1 \wedge \partial_2 + q^1 \partial_2 \wedge \partial_0 + q^2 \partial_0 \wedge \partial_1 \bigr]. \]
Consider distributing the sums out of this expression.
From the basis expression in Equation~\eqref{eqn.snbasic} we see that terms vanish unless
the indices in the triple wedge product are pairwise distinct, so that self-brackets of individual summands vanish.
Furthermore, commutativity of the SN-bracket on bivectors means that the cross terms group in pairs, so we have:
\[ \bigl[ q, q \bigr] = 2 \times \Bigl(
     \bigl[ q^0 \partial_1 \wedge \partial_2,\; q^1 \partial_2 \wedge \partial_0 \bigr] +
     \bigl[ q^1 \partial_2 \wedge \partial_0,\; q^2 \partial_0 \wedge \partial_1 \bigr] +
     \bigl[ q^2 \partial_0 \wedge \partial_1,\; q^0 \partial_1 \wedge \partial_2 \bigr] \Bigr). \]
Now we apply Equation~\eqref{eqn.snbasic} to each term and group the results. Since the
four indices $j$, $k$, $m$, $n$ are always three distinct numbers and $k = m$, only two terms are non-zero, namely
$- g f_{,m} \,\trwdg{j}{k}{n} - f g_{,k} \,\trwdg{m}{j}{n} = \bigl(f g_{,k} - g f_{,k} \bigr) \, \trwdg{j}{k}{n}$.
We find:
\[ \bigl[ q, q \bigr] = 2 \, \trwdg{0}{1}{2} \Bigl(
     q^1 q^2_{,0} - q^2 q^1_{,0}  +  q^2 q^0_{,1} - q^0 q^2_{,1}  +  q^0 q^1_{,2} - q^1 q^0_{,2}\Bigr). \]

A bivector field $q$ is called a \emph{Poisson bivector} if it is integrable,
which happens if and only if its SN-self-bracket vanishes, $[q, q] = 0$.
If $q$ is given in coordinates by Equation~\eqref{eq.defq}, with $q^0, q^1, q^2 \in H^0(U; \mathcal{O})$,
then the integrability condition is:
\begin{align}\label{eq.sn0}
  0 &= q^1 \frac{\partial q^2}{\partial x^0} - q^2 \frac{\partial q^1}{\partial x^0}
  + q^2 \frac{\partial q^0}{\partial x^1} - q^0 \frac{\partial q^2}{\partial x^1}
  + q^0 \frac{\partial q^1}{\partial x^2} - q^1 \frac{\partial q^0}{\partial x^2} \nonumber \\
  &= q^1 \frac{\partial q^2}{\partial z} - q^2 \frac{\partial q^1}{\partial z}
  + q^2 \frac{\partial q^0}{\partial u_1} - q^0 \frac{\partial q^2}{\partial u_1}
  + q^0 \frac{\partial q^1}{\partial u_2} - q^1 \frac{\partial q^0}{\partial u_2}.
\end{align}
Note that by \Cref{subm}, $\mathcal{O}$-multiples of Poisson bivectors are themselves Poisson,
which we can also see directly from the above explicit expressions.

Some Poisson structures on local surfaces will be useful. We summarize a few results.

\begin{remark}[surfaces]\label{Zs}Using  canonical coordinate charts 
$Z_k= \Tot(\mathcal O_{\mathbb P^1}(-k))$
\cite[Lem.\thinspace 2.8]{BG2} calculated all their
 Poisson structures, obtaining generators 
as:
$ (1,-\xi), (z,-1) $  for $k =1$; 
$(1,-1)$ for  $k =2$;
$(u,-\xi^2v), (zu, -\xi v), (z^2u,-v)$ for $k \geq 3$,
written  in the basis $(\partial_z \wedge \partial_u, \partial_\xi \wedge\partial_v ).$
We will also use the generators  for  Poisson structures on $Z_0$ which are $(1,-\xi^2),(z,-\xi ),(z^2,-1)$, and 
 for $Z_{-1}$ which are $(1,-\xi^3),(z,-\xi^2 ),(z^2,-\xi), (z^3,1)$. 
\end{remark}

%

\section{Poisson  structures on \texorpdfstring{$W_1$}{W\_1}}

Let $\imath\colon U \hookrightarrow W_1$ denote the inclusion. We
actually demand that the coefficients of $q$ are functions on all of
$W_1$, i.e.\ that they should be in the image of $\imath^* \colon
R \ce H^0(W_1;\mathcal{O}_{W_1}) \to H^0(U;\mathcal{O}_{U})$.
(We will not distinguish between $R$ and its image over $U$:
we are only working in local coordinates on $U$, but with the understanding
that we are describing global objects on $W_1$.)
In local coordinates on $U$, $R$ consists of convergent power series in
\[ \bigl\{ 1, u_1, z u_1, u_2, z u_2 \bigr\} \text{ .} \]
This imposes additional conditions on the coefficients $q^i$.

\begin{lemma}\label{biW1}
The space  $M_1 = H^0(W_1,\Lambda^2TW_1)$
parametrizing all holomorphic bivector fields on $W_1$ 
has the following structure as a module over global holomorphic functions:
 $$M_1 = \langle e_1,e_2,e_3,e_4 \rangle / \langle zu_2e_1-zu_1 e_2- u_2 e_3 +u_1e_4 \rangle \text{ .}$$
\end{lemma}

\begin{proof}
Since $M_1$ is given by global holomorphic sections of $\Lambda^2TW_1$,
using \v{C}ech cohomology, we search for  $a,b,c$  holomorphic functions on $U$ such that 
$$
\left[
\begin{matrix}
z^2 & -zu_1 & -zu_2 \\
0 & z^{-1} & 0 \\
0 & 0 & z^{-1}
\end{matrix}
\right]
\left[
\begin{matrix}
a \\
b \\
c
\end{matrix}
\right]
$$
is holomorphic on $V$.
To start with $\displaystyle a= \sum_{l=0}^\infty\sum_{i=0}^\infty\sum_{s=0}^\infty a_{lis} z^lu_1^ iu_2^s$
and similar for $b$ and $c$.
Direct calculation by formal neighborhoods of $\mathbb P^1 \subset W_1$ gives the expression 
of the sections. It turns out that all generators we need already appear on the second formal 
neighborhood, where we have: 
$$\left[\begin{matrix}
a \\
b \\
c
\end{matrix}\right]= 
b_{000}
 \left[
\begin{matrix}0 \\ 1\\ 0
\end{matrix}\right]+
c_{000}
 \left[
\begin{matrix}0 \\ 0\\ 1
\end{matrix}\right]+
b_{100}
\left[\begin{matrix} u_1\\ z \\0
\end{matrix}\right]
+
c_{100}
\left[\begin{matrix} u_2\\ 0 \\ z
\end{matrix}\right]
+
a_{020}
 \left[
\begin{matrix} u_1^2\\ 0\\ 0
\end{matrix}\right]+
a_{002}
 \left[
\begin{matrix} u_2^2 \\ 0\\ 0
\end{matrix}\right]+
a_{011}
\left[\begin{matrix} u_1u_2\\ 0 \\0
\end{matrix}\right]
$$ 
$$+b_{010}
 \left[
\begin{matrix}0 \\ u_1\\ 0
\end{matrix}\right]+
b_{110}
 \left[
\begin{matrix}0 \\ zu_1\\ 0
\end{matrix}\right]+
b_{210}
 \left[
\begin{matrix}zu_1^2 \\ z^2u_1\\ 0
\end{matrix}\right]+
b_{001}
 \left[
\begin{matrix}0 \\ u_2\\ 0
\end{matrix}\right]+
b_{101}
 \left[
\begin{matrix}0 \\ zu_2\\ 0
\end{matrix}\right]+
b_{201}
 \left[
\begin{matrix}zu_1u_2 \\ z^2u_2\\ 0
\end{matrix}\right].$$

At this point we have 13  generators of $M_1$ as a vector space over $\mathbb C$:
$$e_1\ce
 \left[
\begin{matrix}0 \\ 1\\ 0
\end{matrix}\right],
e_2\ce
 \left[
\begin{matrix}0 \\ 0\\ 1
\end{matrix}\right],
e_3\ce
\left[\begin{matrix} u_1\\ z \\0
\end{matrix}\right],
e_4\ce
\left[\begin{matrix} u_2\\ 0 \\ z
\end{matrix}\right],
e_5\ce
 \left[
\begin{matrix} u_1^2\\ 0\\ 0
\end{matrix}\right],
e_6\ce
 \left[
\begin{matrix} u_2^2 \\ 0\\ 0
\end{matrix}\right],
e_7\ce
\left[\begin{matrix} u_1u_2\\ 0 \\0
\end{matrix}\right],
$$ 
$$e_8\ce 
 \left[
\begin{matrix}0 \\ u_1\\ 0
\end{matrix}\right],
e_9\ce
 \left[
\begin{matrix}0 \\ zu_1\\ 0
\end{matrix}\right],
e_{10}\ce 
 \left[
\begin{matrix}zu_1^2 \\ z^2u_1\\ 0
\end{matrix}\right],
e_{11}\ce 
 \left[
\begin{matrix}0 \\ u_2\\ 0
\end{matrix}\right],
e_{12}\ce
 \left[
\begin{matrix}0 \\ zu_2\\ 0
\end{matrix}\right],
e_{13}\ce
 \left[
\begin{matrix}zu_1u_2 \\ z^2u_2\\ 0
\end{matrix}\right].$$

These satisfy the set of relations:
$$ zu_2e_1-zu_1e_2-u_2 e_3 +u_1e_4=0, \quad   u_2^2e_5 -  u_1^2e_ 6=0$$
$$e_5-u_1e_3+zu_1e_1=0\quad  e_6-u_2e_4+zu_2e_2=0$$
$$e_7-u_2e_3+zu_2e_1=0\quad  e_7-u_1e_4+zu_1e_2=0$$
$$   u_2e_5 - u_1e_7=0, \quad   u_1e_6- e_7u_2=0.$$
We then proceed to obtain simpler presentations for $M_1$. For instance, clearly the  relations
on lines 2 and 3 
may be used to remove $e_5,e_6,e_7$, simplifying the presentation of $M_1$ to  
a set of 10 generators with 4 relations, and so on. 

After a long series of reductions, or else, using a computer 
algebra, we arrive at a far simpler presentation:
$M_1= \langle e_1,e_2,e_3,e_4 \rangle$ with the single relation. 
$$ zu_2e_1-zu_1 e_2- u_2 e_3 +u_1e_4=0.$$ 
\end{proof}

\begin{theorem}\label{t1}
Every holomorphic Poisson structure on $W_1$ has the form $\sum_{i=1}^4p^ie_i$ where 
$ (p^1,p^2,p^3,p^4) \in \ker B$.
\end{theorem}

Specifically, by \Cref{biW1}
global bivectors on $W_1$ are generated by four elements over $R$,
given on the $U$ chart by
$$e_1=
	\left[\begin{matrix}
		0 \\
		1 \\
  		0
	\end{matrix}\right] ,
e_2=
	\left[\begin{matrix}
		0 \\
		0 \\
  		1
	\end{matrix}\right] ,
e_3=
	\left[\begin{matrix}
		u_1 \\
		z \\
  		0
	\end{matrix}\right],
e_4=
	\left[\begin{matrix}
		u_2 \\
		0 \\
  		z
	\end{matrix}\right].
$$ 
Now write $ p = \sum_{h=1}^4 p^h e_h $
for a bivector field $p$ that extends to all of $W_1$, $p \in H^0(W_1; \Lambda^2 TW_1)$, that is,
\begin{align*}
  p &= p^1 b_1 + p^2 b_2 + p^3 (u_1 b_0 + z b_1) + p^4 (u_2 b_0 + z b_2) \\
    &= (\underbrace{u_1 p^3 + u_2 p^4}_{q^0}) b_0 + (\underbrace{p^1 + z p^3}_{q^1}) b_1 + (\underbrace{p^2 + z p^4}_{q^2}) b_2 \text{ ,}
\end{align*}
where $p^h \in R$ for $h=1, 2, 3, 4$.
We consider the integrability condition \eqref{eq.sn0} with:
\begin{align*}
  q^0(z, u_1, u_2) &= u_1 p^3(z, u_1, u_2) + u_2 p^4(z, u_1, u_2) \\
  q^1(z, u_1, u_2) &= p^1(z, u_1, u_2) + z p^3(z, u_1, u_2) \\
  q^2(z, u_1, u_2) &= p^2(z, u_1, u_2) + z p^4(z, u_1, u_2)
\end{align*}
The condition becomes:
\begin{align}\label{eq.sn0w1}
  0 =& \ (p^1 + z p^3) \frac{\partial(p^2 + z p^4)}{\partial z} - (p^2 + z p^4) \frac{\partial(p^1 + z p^3)}{\partial z} \nonumber \\
    &+ (p^2 + z p^4) \frac{\partial(u_1 p^3 + u_2 p^4)}{\partial{u_1}} - (u_1 p^3 + u_2 p^4) \frac{\partial(p^2 + z p^4)}{\partial{u_1}} \nonumber \\
    &- (p^1 + z p^3) \frac{\partial(u_1 p^3 + u_2 p^4)}{\partial{u_2}} + (u_1 p^3 + u_2 p^4) \frac{\partial(p^1 + z p^3)}{\partial{u_2}} \nonumber \\[1em]
    =& \ (p^1 + z p^3)(p^2_{,0} + p^4 + z p^4_{,0}) - (p^2 + z p^4)(p^1_{,0} + p^3 + z p^3_{,0}) \nonumber \\
    &+ (p^2 + z p^4)(p^3 + u_1 p^3_{,1} + u_2 p^4_{,1}) - (u_1 p^3 + u_2 p^4)(p^2_{,1} + z p^4_{,1}) \nonumber \\
    &- (p^1 + z p^3)(p^4 + u_2 p^4_{,2} + u_1 p^3_{,2}) + (u_1 p^3 + u_2 p^4)(p^1_{,2} + z p^3_{,2}) \nonumber \\[1em]
    =& \ \pbra p120 + z(\pbra p140 + \pbra p320) + z^2(\pbra p340) \nonumber \\
    &+ u_1(\pbra p231 + \pbra p312) + zu_1(\pbra p431) \nonumber \\
    &+ u_2(\pbra p241 + \pbra p412) + zu_2(\pbra p432).
\end{align}\\

\begin{note}
The vectors $\langle \{ e_1, e_2, e_3, e_4 \} \rangle_{\mathbb{C}}$ generate a
submodule of Poisson bivector fields over $R$: any such vector field
is of the form $\mu_1 p e_1 + \mu_2 p e_2 + \mu_3 p e_3 + \mu_4 p e_4$
for some $\mu_1, \mu_2, \mu_3, \mu_4 \in \mathbb{C}$ and $p \in R$,
say, which can be seen to satisfy Equation~\eqref{eq.sn0w1}, with $p^h = \mu_h p$ for $h=1,2,3,4$:
each antisymmetric term $\pbra pijk = \mu_i \mu_j (p\,p_{,k} - p\,p_{,k})$ vanishes.
The connection between the $U$ and the $V$ chart is described above.
\end{note}

For $h=1, 2, 3, 4$, write:
\begin{align}
  p^h(z, u_1, u_2)      &= \sum_{s=0}^{\infty}\sum_{t=0}^{\infty} \sum_{r=0}^{s+t} p^h_{rst} z^r u_1^s u_2^t \label{eq.p_series} \\
  p^h_{,0}(z, u_1, u_2) &= \sum_{s=0}^{\infty}\sum_{t=0}^{\infty} \sum_{r=1}^{s+t} p^h_{rst} r z^{r-1} u_1^s u_2^t \label{eq.p_series_z} \\
  p^h_{,1}(z, u_1, u_2) &= \sum_{s=1}^{\infty}\sum_{t=0}^{\infty} \sum_{r=0}^{s+t} p^h_{rst} s z^r u_1^{s-1} u_2^t \label{eq.p_series_u1} \\
  p^h_{,2}(z, u_1, u_2) &= \sum_{s=0}^{\infty}\sum_{t=1}^{\infty} \sum_{r=0}^{s+t} p^h_{rst} t z^r u_1^s u_2^{t-1} \label{eq.p_series_u2}.
\end{align}
We can substitute these power series expansions into condition
\eqref{eq.sn0w1} and derive conditions on every infinitesimal
neighbourhood, i.e.\ for bounded values of $s$ and $t$:
The restriction to the $n^\text{th}$ infinitesimal neighbourhood
sets to zero all terms $u_1^s u_2^t$ for which $s + t > n$.
Note that the series for $p^h$ has $(n + 1)^2$ terms on the $n^\text{th}$ infinitesimal neighbourhood
(more precisely: in the kernel of $\mathcal{O}_{\ell^{(n)}} \to \mathcal{O}_{\ell^{(n - 1)}}$),
where $n = s + t$.

\begin{note}
The expression in Equation~\eqref{eq.sn0w1} is an element of $R$,
i.e.\ a globally holomorphic function. This is to be expected, since
$p = \sum_{h=1}^4 p^h e_h$ is (the restriction to $U$ of) a global
bivector field, and the NS-bracket maps global (multi)vector fields to
global (multi)vector fields (being a composition of Lie brackets,
which map vector fields to vector fields). We can also verify this in
local coordinates: Let $[p, p] = f(p^i, p^i_{,j}) \; \partial_0 \wedge
\partial_1 \wedge \partial_2$, so that $f$ is the right-hand side of
Equation~\eqref{eq.sn0w1}. Note that $\partial_0 \wedge \partial_1
\wedge \partial_2 = \partial_{\tilde{0}} \wedge \partial_{\tilde{1}}
\wedge \partial_{\tilde{2}}$ on $U \cap V$ (after all, $W_1$ is
Calabi-Yau); we show that $f$ is globally holomorphic on $W_1$: If
$p^h \in R$, then $p^h_{,0}$ and $zp^h_{,0}$ are in $R$, too, as is
clear from considering \eqref{eq.p_series_z}.  Terms $u_1 p^h_{,1}$
and $u_2 p^h_{,1}$ are holomorphic, as can be seen from
\eqref{eq.p_series_u1}, and similarly for $u_1 p^h_{,2}$ and $u_2
p^h_{,2}$. The remaining terms are not individually globally
holomorphic, but they group as follows:
\[ p^3 \Bigl( z^2 p^4_{,0} - z u_1 p^4_{,1} - z u_2 p^4_{,2} \Bigr) -
   p^4 \Bigl( z^2 p^3_{,0} - z u_1 p^3_{,1} - z u_2 p^3_{,2} \Bigr) .\]
By considering \eqref{eq.p_series_z}, \eqref{eq.p_series_u1}, and
\eqref{eq.p_series_u2}, we see that the only non-holomorphic terms are
$z^{s + t + 1}u_1^s u_2^t$, and those appear with coefficient
$\bigl((s + t) - s - t\bigr)\bigl(p^4_{s+t,s,t} - p^3_{s+t,s,t}\bigr) = 0$.
\end{note}

\begin{note}\label{operator}
The quasi-linear differential operator $B$ defined above can be written as follows:
\begin{multline*}
  B(p^1, p^2, p^3, p^4) = \mathbf{p}^T Q \, \mathbf{p} = \\
  \Bigl[ p^1 \ p^2 \ p^3 \ p^4 \Bigr]
  \begin{bmatrix}
    0                              & \partial_0      & -u_1 \partial_2                                       & z \partial_0 - u_2 \partial_2 \\
    -\partial_0                    & 0               & -z \partial_0 + u_1 \partial_1                        & u_2 \partial_1 \\
    u_1\partial_2                  & z\partial_0 -u_1\partial_1    & 0                                                     & z^2 \partial_0 - zu_1 \partial_1 - zu_2 \partial_2 \\
    -z \partial_0 + u_2 \partial_2 & -u_2 \partial_1 & -z^2 \partial_0 + z u_1 \partial_1 + z u_2 \partial_2 & 0
  \end{bmatrix}
  \begin{bmatrix}p^1 \\ p^2 \\ p^3 \\ p^4 \end{bmatrix}
\end{multline*}
where have expressed $f$ using the quadratic form $Q$. We may linearize
this differential equation around a fixed solution $\mathbf{p}^T = [p^1 \ p^2 \ p^3 \ p^4]$:
\[ \lim_{\varepsilon \to 0} \frac{1}{\varepsilon} \Bigl( f(\mathbf{p} + \varepsilon \Delta\mathbf{p}) - f(\mathbf{p}) \Bigr) =
    \Delta\mathbf{p}^T Q \, \mathbf{p} + \mathbf{p}^T Q \, \Delta\mathbf{p} .\]
\end{note}

\subsection{Symmetries and embeddings}

We now give isomorphism between the Poisson structures on $W_1$.

\begin{remark} Note that there are two clear symmetries of $W_1$:
\begin{itemize}
\item exchanging the radial directions $s_0(z,u_1,u_2) = (z,u_2,u_1)$ and
\item exchanging the  charts $U$ and $V$, that is, $s_1(z,u_1,u_2)= (\xi,v_1,v_2).$
\end{itemize}
\end{remark}

These symmetries are automorphisms of $W_1$ and are also Poisson isomorphisms between some structures 
on $W_1$ as shown in the diagram below:

\begin{center}
\begin{tikzcd}
(W_1,e_3)		\arrow[<->]{r}{s_1}
					\arrow[swap,<->]{d}{s_0}
& 
(W_1,e_1)		\arrow[<->]{d}{s_0} 
\\
(W_1,e_2)	\arrow[swap,<->]{r}{s_1}
&
(W_1,e_4)
\end{tikzcd}.
\end{center}

In other words, we have that 
$e_1= s_0^*(e_2)$, $e_4= s_1^*(e_2)$, and $e_4 =s_1^* s_0^* e_2$.

\begin{theorem}\label{iso}
The Poisson structures $e_1,e_2,e_3,e_4$ are all  pairwise isomorphic.
\end{theorem}

There are 2 obvious inclusions of the surface $Z_1$ into the threefold $W_1$. 

\begin{notation}\label{pis} We denote by $\pi_{i}$ the Poisson structure on $Z_k$ 
that is given on the $U$-chart by $z^iu^\varepsilon$ where 
$\varepsilon = 0$ if  $i \leq 2$ and $\varepsilon = 1$ if $i \geq 3$.
We denote by $j_1\colon Z_1 \rightarrow W_1$ (resp. $j_2$) the inclusion into the first  (resp. second) fiber coordinate, that is, 
on the $U$-chart  $j_1(z,u) = (z,u,0)$ (resp.  $j_2(z,u) = (z,0,u)$). 
We call $j_1 $, $s_0j_1, s_1 j_1, s_1s_0j_1$ the {\bf principal embeddings} of $Z_1$ into $W_1$.
\end{notation}

\begin{theorem}\label{princ}
The 4 principal  embeddings of the Poisson surface $(Z_1,\pi_0)$ generate all Poisson structures on $W_1$.
\end{theorem}

\begin{proof}
Let $j_1(Z_1)$ (resp. $j_2(Z_1)$) be the embedding of the surface $Z_1$  into  $W_1$ 
 by $u_2=0$ and $v_2=0$ (resp. $u_1=0$ and $v_1=0$). Then  Poisson structures induced by the first embedding are:
 $$(j_1)_*(1)_U= \left[\begin{matrix}0 \\ 0 \\ 1\end{matrix}\right]_U,\quad 
(j_1)_*(-\xi)_V= \left[\begin{matrix}0 \\ 0 \\ -\xi\end{matrix}\right]_V, \quad \text{hence} \quad 
e_2\vert_{j_1(Z_1)}=(j_1)_*\pi_0, $$
analogously, 
$ \gamma_1\vert_{j_1(Z_1)}=j_1(s \pi_0) $.
The   induced Poisson structures by the second embedding are 
$$(j_2)_*(1)_U= \left[\begin{matrix}0 \\ 1 \\ 0\end{matrix}\right]_U, \quad
(j_2)_*(-\xi)_V= \left[\begin{matrix}0 \\ -\xi\\ 0\end{matrix}\right]_V, \quad  \text{hence} \quad 
 e_3\vert_{j_2(Z_1)}=(j_2)_*\pi_0,$$
analogously
$e_1\vert_{j_2(Z_1)} =(j_2)_*(s\pi_0) $.

\end{proof}

\subsection{Symplectic foliations on \texorpdfstring{\except{toc}{$\bm{W_1}$}\for{toc}{$W_1$}}{W\_1}}\label{leaf1}

Since  $e_1,e_2,e_3,e_4$ are all isomorphic, to understand their corresponding 
symplectic foliations, it is enough to describe the symplectic foliation in one case.
We consider  $e_2$ whose expression in canonical coordinates is
$\{f,g\}_{e_2} =  (df\wedge dg)  \lrcorner  (\frac{\partial}{\partial z} \wedge \frac{\partial}{\partial u_1})=
 (df\wedge dg)  \lrcorner  (\partial_0 \wedge \partial_1).$ 

\begin{theorem}\label{w1fol}
The symplectic foliation for $(W_1, \partial_0 \wedge \partial_1)$ is given by:

\begin{itemize}
\item $ \partial_0 \wedge \partial_1$ has degeneracy locus on the line 
$\{v_2= \xi=0\}$, where the leaves are $0$-dimensional, consisting of single points, and 
\item
$2$-dimensional symplectic leaves cut out on the $U$ chart by $u_2$ constant.
\end{itemize}
\end{theorem}

\begin{proof}
To find the symplectic leaves we compute Poisson cohomology $H^0(W_1,e_2)$, 
obtaining that $ e_2 =  \, \partial_0 \wedge \partial_1$
has 2 dimensional symplectic leaves cut out on the $U$ chart by $u_2$ constant  (the Casimir functions), 
and
next changing coordinates
$$
\left[
\begin{matrix}
z^2 & -zu_1 & -zu_2 \\
0 & z^{-1} & 0 \\
0 & 0 & z^{-1}
\end{matrix}
\right]
\left[
\begin{matrix}
0 \\0 \\  1
\end{matrix}
\right]= \left[
\begin{matrix}
-zu_2 \\
0 \\
 z^{-1}
\end{matrix}
\right]= 
 \left[
\begin{matrix}
-v_2 \\
0 \\
 \xi
\end{matrix}
\right]
$$
so, we see that the expression of $e_2$ on the $V$-coordinate 
is $-v_2\frac{\partial}{\partial v_1}\wedge \frac{\partial}{\partial v_2}+\xi \frac{\partial}{\partial \xi} \wedge \frac{\partial}{\partial  v_1}$, which 
vanishes when $\xi=v_2=0$.
 Hence
 $e_2$ has degeneracy locus on the line 
$D(e_2)=\{v_2= \xi=0\}$, where the leaves are
 0 dimensional,
consisting of each of the points in the line $\xi=v_2=0$.
\end{proof}

\section{Poisson structures on \texorpdfstring{$W_2$}{W\_2}}

The Calabi--Yau threefold we consider in this section is
 $$W_2\ce \Tot (\mathcal O_{\mathbb P^1}(-2) \oplus \mathcal O_{\mathbb P^1}) = Z_2 \otimes \mathbb C.$$
We will carry out calculations using the canonical coordinates $W_1= U \cup V$ 
where $U \simeq \mathbb C^3 \simeq V$ with coordinates $U = \{z,u_1,u_2\}$, $V= \{\xi, v_1,v_2\}$,
and change of coordinates on $U\cap V \simeq \mathbb C^* \times \mathbb C\times \mathbb C$ given by 
\[ \bigl\{ \xi = z^{-1} \text{ ,\quad}
         v_1 = z^2 u_1 \text{ ,\quad}
         v_2 =  u_2 \bigr\} \text{ ,} \]
so that $z = \xi^{-1}$, $u_1 = \xi^2 v_1$, and $u_2 =  v_2$.

The transition matrix for the tangent bundle  is the Jacobian matrix of the change of 
coordinates, and taking the second exterior power we obtain the transition matrix for $\Lambda^2TW_2$:
\[
\left[
\begin{matrix}
z^2 & -2zu_1 & 0 \\
0 & -z^{-2} & 0 \\
0 & 0 & -1
\end{matrix}
\right]
.\]

Let $\imath\colon U \hookrightarrow W_2$ denote the inclusion. We
actually demand that the coefficients of $q$ are functions on all of
$W_2$, i.e.\ that they should be in the image of $\imath^* \colon
R \ce H^0(W_2;\mathcal{O}_{W_2}) \to H^0(U;\mathcal{O}_{U})$.
(We will not distinguish between $R$ and its image over $U$:
we are only working in local coordinates on $U$, but with the understanding
that we are describing global objects on $W_2$.)
In local coordinates on $U$, $R$ consists of convergent power series in
\[ \bigl\{    1, u_1,zu_1,z^2u_1,u_2 \bigr\} \text{ .} \]
Now write $p = \sum_{h=1}^5 p^h e_h$
for a bivector field $p$ that extends to all of $W_2$, $p \in H^0(W_2; \Lambda^2 TW_2)$.

\begin{lemma} \label{W2gens}The space  $M_2 = H^0(W_2,\Lambda^2TW_2)$
parametrizing all holomorphic bivector fields  on $W_2$ 
has the following structure as a module over global holomorphic functions:
$$M_2 = \langle e_1,e_2,e_3,e_4, e_5 \rangle / \langle  u_1e_3-zu_1e_1,
 u_2e_5-zu_2e_3-2zu_2e_2\rangle.$$
\end{lemma}

\begin{proof}

To find 
 $H^0(W_2,\Lambda^2TW_2)$
 we need global holomorphic sections,  that is, we must find $a,b,c$  holomorphic on $U$ such that 
$$
\left[
\begin{matrix}
z^2 & -2zu_1 & 0 \\
0 & -z^{-2} & 0 \\
0 & 0 & -1
\end{matrix}
\right]
\left[
\begin{matrix}
a \\
b \\
c
\end{matrix}
\right]
$$
is holomorphic on $V$.
To start with $\displaystyle a= \sum_{l=0}^\infty\sum_{i=0}^\infty\sum_{s=0}^\infty a_{lis} z^lu_1^ iu_2^s$
and similar for $b$ and $c$.
     We proceed by calculations on formal neighborhoods of the $\mathbb P^1 \subset W_2$ 
      and verify that generators for all global sections are already found on the first formal neighborhood,
      where the general expression of a section of $\Lambda^2TW_2$ is:
$$
\left[\begin{matrix}
		a \\
		b \\
  		c
	\end{matrix}\right] =
c_{000}
	\left[\begin{matrix}
		0 \\
		0 \\
  		1
	\end{matrix}\right] 
+c_{010}
	\left[\begin{matrix}
		0 \\
		0 \\
  		u_1
	\end{matrix}\right]
+c_{110}
	\left[\begin{matrix}
		0 \\
		0 \\
  		zu_1
	\end{matrix}\right]
+c_{210}
	\left[\begin{matrix}
		0 \\
		0 \\
  		z^2u_1
	\end{matrix}\right]
+c_{001}
	\left[\begin{matrix}
		0 \\
		0 \\
  		u_2
	\end{matrix}\right]
+a_{010}
	\left[\begin{matrix}
		u_1 \\
		0 \\
  		0
	\end{matrix}\right] $$
	$$
+b_{000}
	\left[\begin{matrix}
		0 \\
		1 \\
  		0
	\end{matrix}\right] 
+b_{100}
	\left[\begin{matrix}
		0 \\
		z \\
  		0
	\end{matrix}\right]
+b_{200}
	\left[\begin{matrix}
		2 zu_1 \\
		z^2 \\
  		0
	\end{matrix}\right]
+b_{010}
	\left[\begin{matrix}
		0 \\
		u_1 \\
  		0
	\end{matrix}\right]\\
+b_{110}
	\left[\begin{matrix}
		0 \\
		zu_1 \\
  		0
	\end{matrix}\right]
+b_{210}
	\left[\begin{matrix}
		0 \\
		z^2u_1 \\
  		0
	\end{matrix}\right] \\
+b_{310}
	\left[\begin{matrix}
		0 \\
		z^3u_1 \\
  		0
	\end{matrix}\right].
     $$     
We  then need the structure of    $M = H^0(W_2,TW_2)$ as a module over global functions.
At first this gives us  potentially 13 generators, but since   $u_1,zu_1,z^2u_1,u_2$
are global functions, we obtain that in fact all sections can be obtained from the smaller set of
     generators:
$$
e_1=
	\left[\begin{matrix}
		0 \\
		1 \\
  		0
	\end{matrix}\right], 
e_2=
	\left[\begin{matrix}
		u_1 \\
		0 \\
  		0
	\end{matrix}\right] ,	
e_3=
\left[\begin{matrix}
		0 \\
		z \\
  		0
	\end{matrix}\right],
e_4=\left[\begin{matrix}
		0 \\
		0 \\
  		1
	\end{matrix}\right] ,
e_5=
	\left[\begin{matrix}
		2zu_1 \\
		z^2 \\
  		0
	\end{matrix}\right].
$$     
To describe the module structure over global sections, 
we write  relations among the generators. 
We have the equations:
$$e_3-ze_1=0$$
$$e_5-ze_3-2ze_2 =0.$$
Note that there are no equations involving $e_4$. This corresponds to the fact that the geometry of $W_2= Z_2 \times \mathbb C$
is that of a surface product $\mathbb C$. Accordingly, we shall not involve $u_2$ in the relations to be obtained from these equations. 
To get relations as an $\mathcal O(W_2)$-module, we multiply the equations by $u_1$, obtaining:
$$u_1e_3-zu_1e_1=0$$
$$u_2e_5-zu_3e_4-2zu_2e_2 =0.$$
\end{proof}

Next we discuss which of these bivector fields give isomorphic Poisson structures.

\begin{lemma}\label{iso1-5}
The Poisson manifolds $(W_2, e_1)$ and $(W_2,e_5)$ are isomorphic.
\end{lemma}

\begin{proof}
Note that by writing $e_1$ in $V$-coordinates we get:
\[
\left[
\begin{matrix}
z^2 & -2zu_1 &  0 \\
0 &- z^{-2} & 0 \\
0 & 0 & -1
\end{matrix}
\right]
\left[
\begin{matrix}
0 \\ 1 \\ 0
\end{matrix}
\right]
=
\left[
\begin{matrix}
-2zu_1 \\ -z^{-2} \\ 0
\end{matrix}
\right]
=
\left[
\begin{matrix}
-2\xi v_1 \\ -\xi^2 \\ 0
\end{matrix}
\right].
\]
So we get an isomorphism between $(W_2, e_1)$ and $(W_2,-e_5)$ by mapping the $U$-chart of one to the $V$-chart of the other and vice-versa. Then the desired isomorphism follows from the fact that $(W_2,e_5)$ and $(W_2,-e_5)$ are isomorphic.
\end{proof}

We now describe the loci where Poisson structures on $W_2$ degenerate.

\begin{lemma} \label{degeneracy2}
The degeneracy loci of Poisson structures on $W_2$ are:  
\begin{itemize}
\item isomorphic to $\mathbb C^2$ for $e_1$, and
\item isomorphic to $\mathbb P^1 \times \mathbb C$ for $e_2$, 
\item isomorphic to $\mathbb C^2\cup \mathbb C$  for $e_3$, and
\item empty for $e_4$. 
\end{itemize}
\end{lemma}

\begin{proof} The coefficients of the Poisson structures in coordinate charts are:
$$
e_1 \ce   \left[\begin{matrix} 0 \\ 1 \\ 0 \end{matrix}  \right] _U =  \left[\begin{matrix}  -2 \xi v_1 \\ -\xi^2 \\ 0  \end{matrix} \right]  _V 
e_2 \ce \left[\begin{matrix}  u_1 \\ 0 \\ 0 \end{matrix} \right] _U = \left[ \begin{matrix}  v_1 \\ 0 \\ 0  \end{matrix}  \right] _V
e_3  \ce \left[\begin{matrix}  0 \\ z \\ 0 \end{matrix}  \right] _U =  \left[\begin{matrix}  -2 v_1 \\ \xi \\ 0  \end{matrix}  \right] _V 
e_4 \ce   \left[ \begin{matrix}  0 \\ 0 \\ 1 \end{matrix}  \right] _U = \left[ \begin{matrix} 0 \\ 0 \\ -1  \end{matrix}  \right] _V.
$$

On the $V$ chart we have that $e_1$ degenerates when $\xi=0$ which is copy of $\mathbb C^2$.
Therefore, we have that $e_2$ degenerates when $u_1=v_1=0$ which gives a trivial product $\mathbb P^1 \times \mathbb C$.
On the $U$ chart we have that $e_3$ degenerates when $z=0$ which is a copy of $\mathbb C^2$, on the $V$ chart we have that $e_3$ degenerates when $\xi=v_1=0$ which is copy of $\mathbb C$.
For $e_4$ the degeneracy locus is empty.
\end{proof}

\begin{corollary} The  brackets $e_1, e_2, e_3, e_4$ give $W_2$ non-isomorphic Poisson structures. 
\end{corollary}

There are  natural inclusions of the surfaces $\mathbb C$, $Z_0$, and $Z_2$ into $W_2$:

\begin{notation}
We denote by $j_s $ for $s=0,1,2$ the inclusions of $\mathbb C^2$, $Z_0$, and $Z_2$ into the threefold $W_2$.
Hence, in coordinates we have: 
\begin{itemize}
\item   $j_0\colon \mathbb C^2 \rightarrow W_2$ includes $\mathbb C^2$ as the fiber  $z=\xi=1$,
\item $j_1\colon Z_2 \rightarrow W_2$ includes $Z_2$ as the surface $u_2=v_2=0$, 
\item $j_2\colon Z_0 \rightarrow W_2$ includes $Z_0$ as the surface $u_1=\xi^2 v_1=0$.
\end{itemize}
\end{notation}

\begin{theorem} \label{emb2}The embedded Poisson surfaces $j_0(\mathbb C^2, \pi_0),$ $j_1(Z_2,\pi_0)$ and $j_2(Z_0,\pi_i)$ 
with $i=0,1,2,$
 generate 
all Poisson structures on $W_2$.
\end{theorem}

\begin{proof}
Let $j_1(Z_2)$ (resp. $j_2(Z_0)$) be the embedding of the surface $Z_2$ (resp. $Z_0$) into  $W_2$ 
cut out by $u_2=v_2=0$ (resp. $u_1=\xi^2v_1=0$). 
 Then Poisson structure induced by the first embedding is:
 $$(j_1)_*(1)_U= \left[\begin{matrix}0 \\ 0 \\ 1\end{matrix}\right]_U,\quad 
(j_1)_*(-\xi)_V= \left[\begin{matrix}0 \\ 0 \\ -1\end{matrix}\right]_V, \quad \text{hence} \quad 
e_4\vert_{j_1(Z_2)}=(j_1)_*\pi_0. $$
The Poisson structures induced  by the second embedding are
$$(j_2)_*(1)_U = \left[\begin{matrix}0 \\ 1 \\ 0\end{matrix}\right]_U, \quad 
(j_2)_*(-\xi^2)_V = \left[\begin{matrix}0 \\ -\xi^2\\ 0\end{matrix}\right]_V, \quad \text{hence} \quad
 e_1\vert_{j_2(Z_0)}=(j_2)_*\pi_0,$$
and analogously
$e_2\vert_{j_2(Z_0)} =(j_2)_*(s\pi_1)$.
Since  $j_0$ has image at  $u=\xi=1$, we obtain

$(j_0)_*(1)_U=   \left[\begin{matrix}1 \\ 0 \\ 0\end{matrix}\right]_U , \quad  (j_0)_*(1)_V= \left[\begin{matrix}1 \\ 0 \\ 0\end{matrix}\right]_V 
\Leftrightarrow  (j_0)_*(z)_U=   \left[\begin{matrix}u_1 \\ 0 \\ 0\end{matrix}\right]_V , \quad  (j_0)_*(\xi)_V= \left[\begin{matrix}v_1 \\ 0 \\ 0\end{matrix}\right]_V ,$

\noindent hence 
 $ e_3\vert_{j_0(Z_0)}=(j_0)_*\pi_0.$
\end{proof}

\subsection{Symplectic foliations on \texorpdfstring{$\bm{W_2}$}{W\_2}}

In this section we perform the cohomological calculations, and identify the leaves of the symplectic foliation associated 
to each Poisson structure on $W_2$. 

\begin{lemma}\label{coh-beta0}
$H^0(W_2, e_1) = \{ f \in \mathcal{O}(W_2) / f = f(u_1) \}$
\end{lemma}

\begin{proof}
Recall that $e_1 = - \partial_0 \wedge \partial_2$. Then we have 
\[
[f, e_1] = -[f, \partial_0 \wedge \partial_2] = -[f, \partial_0] \wedge \partial_2 + \partial_0 \wedge [f, \partial_2] =
 \dfrac{\partial f}{\partial u_2} \partial_0 - \dfrac{\partial f}{\partial z} \partial_2,
\]
so that $ f \in \ker [e_1, \cdot] $ if and only if 
$\displaystyle 
\dfrac{\partial f}{\partial z} = \dfrac{\partial f}{\partial u_2} = 0,
$
i.e., $f$ does not depend on $z$ and $u_2$.
\end{proof}

\begin{lemma}\label{coh-alpha}
$H^0(W_2, e_2 ) = \{ f \in \mathcal{O}(W_2) / f = f(z) \}$
\end{lemma}

\begin{proof}
Recall that $e_2 = u_1 \partial_1 \wedge \partial_2$. Then we have 
\[
[f, e_2] = [f, u_1 \partial_1 \wedge \partial_2] = [f, u_1 \partial_1] \wedge \partial_2 - u_1 \partial_1 \wedge [f, \partial_2] = 
u_1 \dfrac{\partial f}{\partial u_1} \partial_2 - u_1\dfrac{\partial f}{\partial u_2} \partial_1, =
\]
so that  $ f \in \ker [e_3, \cdot] $ if and only if 
$\displaystyle
\dfrac{\partial f}{\partial u_1} = \dfrac{\partial f}{\partial u_2} = 0, 
$
i.e., $f$ does not depend on $u_1$ and $u_2$. 
\end{proof}

\begin{lemma}\label{coh-beta1}
$H^0(W_2, e_3) = \{ f \in \mathcal{O}(W_2) / f = f(u_1) \}$
\end{lemma}

\begin{proof}
Recall that $e_3 = - z\partial_0 \wedge \partial_2$. Then we have 
\[
[f, e_3] = -[f, -z \partial_0 \wedge \partial_2] = -[f, z\partial_0] \wedge \partial_2 + z\partial_0 \wedge [f, \partial_2] = z\dfrac{\partial f}{\partial u_2} \partial_0 - z\dfrac{\partial f}{\partial z} \partial_2,
\]
so  that  $ f \in \ker [e_3, \cdot] $ if and only if 
$\displaystyle
\dfrac{\partial f}{\partial z} = \dfrac{\partial f}{\partial u_2} = 0,
$
i.e., $f$ does not depend on $z$ and $u_2$. 
\end{proof}

\begin{lemma}\label{coh-gamma}
$H^0(W_2, e_4) = \{ f \in \mathcal{O}(W_2) / f = f(u_2) \}$.
\end{lemma}

\begin{proof}
Recall that $e_4 = \partial_0 \wedge \partial_1$. Then we have
\[
[f, e_4] = [ f, \partial_0 \wedge \partial_1 ] = [f, \partial_0] \wedge \partial_1 - \partial_0 \wedge [f, \partial_1] =
 \dfrac{\partial f}{\partial z} \partial_1 - \dfrac{\partial f}{\partial u_1} \partial_0, 
\]
so that $f \in \ker [e_4, \cdot]$ if and only if 
$\displaystyle
\dfrac{\partial f}{\partial z} = \dfrac{\partial f}{\partial u_1} = 0,
$
i.e., $f$ does not depend on $z$ and $u_1$.
\end{proof}

We then obtain the description of the symplectic foliations on $W_2$ determined by these Poisson structures.

\begin{theorem}\label{foliations2}
The symplectic foliations on $W_2$ 
have $0$-dimensional leaves consisting of single points over each 
of their corresponding  degeneracy loci described in \Cref{degeneracy2},
and their generic leaves, which are $2$-dimensional, are as follows:
\begin{itemize}
\item surfaces of constant $u_1$  for $e_1$ and $e_3$, one of them isomorphic to $\mathbb P^1 \times \mathbb C$.
\item isomorphic to $\mathbb C-\{0\} \times \mathbb C$ for $e_2$ (contained in the fibers of the projection to $\mathbb P^1$).
\item isomorphic to the surface $Z_2$ and cut out by $u_2=v_2$ constant for $e_4$.
\end{itemize}
\end{theorem}

\vspace{3mm}

\section{Poisson structures on \texorpdfstring{$W_3$}{W\_3}}

The Calabi--Yau threefold we consider in this section is
 $W_3\ce \Tot (\mathcal O_{\mathbb P^1}(-3) \oplus \mathcal O_{\mathbb P^1}(1))$.
We will carry out calculations using the canonical coordinates $W_1= U \cup V$ 
with $U \simeq \mathbb C^3 \simeq V$ with coordinates $U = \{z,u_1,u_2\}$ and $V= \{\xi, v_1,v_2\}$
with change of coordinates on $U\cap V \simeq \mathbb C^* \times \mathbb C\times \mathbb C$ given by 
\[ \bigl\{ \xi = z^{-1} \text{ ,\quad}
         v_1 = z^3 u_1 \text{ ,\quad}
         v_2 =  z^{-1}u_2 \bigr\} \text{ ,} \]
so that $z = \xi^{-1}$, $u_1 = \xi^3 v_1$, and $u_2 =  \xi^{-1} v_2$.

In these coordinates, the transition matrix for the tangent bundle of $W_3$ is the Jacobian matrix of the change of 
coordinates, and taking $\Lambda^2$ we obtain the transition matrix for the second exterior power of the tangent bundle:
\[
\left[
\begin{matrix}
z^2 & -3zu_1 & zu_2 \\
0 & -z^{-3} & 0 \\
0 & 0 & -z
\end{matrix}
\right]
.\]

Let $\imath\colon U \hookrightarrow W_3$ denote the inclusion. We
actually demand that the coefficients of $q$ are functions on all of
$W_3$, i.e.\ that they should be in the image of $\imath^* \colon
R \ce H^0(W_3;\mathcal{O}_{W_3}) \to H^0(U;\mathcal{O}_{U})$.
(We will not distinguish between $R$ and its image over $U$:
we are only working in local coordinates on $U$, but with the understanding
that we are describing global objects on $W_2$.)
In local coordinates on $U$, $R$ consists of convergent power series in
\[ \bigl\{ 1, u_1,zu_1,z^2u_1, z^3u_1,u_1u_2, zu_1u_2, z^2u_1u_2   \bigr\} \text{ .} \]

Holomorphic Poisson structures on $W_3$ are parametrized by elements of
 $M_3\ce H^0(W_2,\Lambda^2TW_3)$, which is  infinite dimensional as a vector 
 space over $\mathbb C$. We will describe the structure of $M_3$ as 
 a module over global functions.

\begin{lemma}\label{W3gens} The space  $M_3 = H^0(W_3,\Lambda^2TW_3)$
parametrizing all holomorphic bivector fields on $W_3$ 
has the following structure as a module over global holomorphic functions:
$$M_3= \mathbb C<e_1, \dots, e_{13}
>/R 
$$
     with the set of relations $R$ given by    
         $$\begin{array}{l}
u_1  e_2  - u_1u_2 e_1 \\
u_1  e_{10} -u_1 u_2 e_3 \\
u_1e_{13} - u_1u_2e_7 \\
\end{array}
\quad\quad \begin{array}{l}
zu_1e_{12} - u_1u_2 e_6 \\
zu_1e_{13} - u_1u_2e_8  \\
u_1e_{11} - zu_1 e_{10} \\
\end{array}
\quad \begin{array}{l}
u_1 e_4 - z u_1e_3   \\
 u_1  e_5 - zu_1e_4  \\
u_1e_8 - zu_1e_7  \\
\end{array} \quad 
\begin{array}{l}
u_1 e_6 - zu_1 e_5 - 3z^2u_1e_1 \\
u_1e_9 - zu_1e_8 + zu_1e_2 \\
u_1e_{12} - zu_1e_{11} - 3zu_1e_1 .\\
\end{array}$$  
\end{lemma}

\begin{remark}
There is no natural  way to simplify the presentation of $M_3$, 
in fact, computer algebra calculations (for example in Macaulay2) 
 also give the same expression for the   minimal presentation of $M_3$. 
So, we really need all 13 generators and 13 relations to describe the space of  Poisson structures 
on $W_3$ as a module over global functions. As a complex vector space it is  infinite dimensional.
\end{remark}

\begin{proof}[Proof of \Cref{W3gens}]
To find $H^0(W_3,\Lambda^2TW_3)$
we need global holomorphic sections  so that we must find $a,b,c$  holomorphic on $U$ such that 
$$
\left[
\begin{matrix}
z^2 & -3zu_1 & zu_2 \\
0 & -z^{-3} & 0 \\
0 & 0 & -z
\end{matrix}
\right]
\left[
\begin{matrix}
a \\
b \\
c
\end{matrix}
\right]
$$
is holomorphic on $V$.
To start with 
$\displaystyle a= \sum_{l=0}^\infty\sum_{i=0}^\infty\sum_{s=0}^\infty a_{lis} z^lu_1^ iu_2^s$ 
and similar for $b$ and $c$.     
     
     We will give a presentation of $M\ce H^0(W_3,\Lambda^2TW_3)$ as a module over global sections. 
     Here we will need to perform calculations up to at least neighborhood 2, unlike the case of $W_2$ 
     where neighborhood 1 was enough. 
     Thus, to calculate the module structure here, 
     we need   the expressions of  sections on the second formal neighborhood, which 
     consist of linear combinations of the following 42 terms:
\begin{flalign*}
	\left[\begin{matrix}
		0 \\
		0 \\
  		z^lu_1
	\end{matrix}\right]_{0\leq l\leq 1}; 
	\left\{	\left[\begin{matrix}
		0 \\
		z^l \\
  		0
	\end{matrix}\right]; 	
	\left[\begin{matrix}
		0 \\
		z^lu_2 \\
  		0
	\end{matrix}\right] \right\}_{0\leq l\leq 2};	
	\left\{ \left[\begin{matrix}
		z^lu_1^2 \\
		0 \\
  		0
	\end{matrix}\right]; 	
	\left[\begin{matrix}
		0 \\
		0 \\
  		z^lu_1^2
	\end{matrix}\right];
		\left[\begin{matrix}
		0 \\
		z^lu_1u_2 \\
  		0
	\end{matrix}\right] \right\}_{0\leq l\leq 4};		
	\left[\begin{matrix}
		0 \\
		z^{l}u_1 \\
  		0
	\end{matrix}\right]_{0\leq l\leq 5};	
	\left[\begin{matrix}
		0 \\
		z^lu_1^2 \\
  		0
	\end{matrix}\right]_{0\leq l\leq 8}; &&
\end{flalign*}
\begin{flalign*}
\left[\begin{matrix}
		u_1 \\
		0 \\
  		0
	\end{matrix}\right];  
\left[\begin{matrix}
		u_1u_2 \\
		0 \\
  		0
	\end{matrix}\right];
\left[\begin{matrix}
		0 \\
		0 \\
  		u_1u_2
	\end{matrix}\right];		
\left[\begin{matrix}
		3z ^2u_1 \\
		z^3 \\
  		0
	\end{matrix}\right]; 	
		\left[\begin{matrix}
		3zu_1u_2 \\
		z^2u_2 \\
  		0
	\end{matrix}\right];	
	\left[\begin{matrix}
		3z^5u_1^2 \\
		z^{6}u_1 \\
  		0
	\end{matrix}\right]; 
		\left[\begin{matrix}
		3z^8u_1^3 \\
		z^9u_1^2 \\
  		0
	\end{matrix}\right];
	\left[\begin{matrix}
		3z^4u_1^2u_2 \\
		z^5u_1u_2 \\
  		0
	\end{matrix}\right];	
		\left[\begin{matrix}
		-zu_1u_2 \\
		0 \\
  		z^2u_1
	\end{matrix}\right];
		\left[\begin{matrix}
		-u_1u_2^2 \\
		0 \\
  		zu_1u_2
	\end{matrix}\right] .&&
\end{flalign*}
\vspace{3mm}

But, upon removing all vectors that can be obtained from others by multiplying by 
a global function we reduce the expression of a global section  to:
$$
\left[\begin{matrix}
		a \\
		b \\
  		c
	\end{matrix}\right] = 
a_{010}
	\left[\begin{matrix}
		u_1 \\
		0 \\
  		0
	\end{matrix}\right]  
+a_{011}
	\left[\begin{matrix}
		u_1u_2 \\
		0 \\
  		0
	\end{matrix}\right] 
+b_{000}
	\left[\begin{matrix}
		0 \\
		1 \\
  		0
	\end{matrix}\right] 
+b_{100}
	\left[\begin{matrix}
		0 \\
		z \\
  		0
	\end{matrix}\right]
+b_{200}
	\left[\begin{matrix}
		0 \\
		z^2 \\
  		0
	\end{matrix}\right]+	
b_{300}
	\left[\begin{matrix}
		3z ^2u_1 \\
		z^3 \\
  		0
	\end{matrix}\right]	
+b_{001}
	\left[\begin{matrix}
		0 \\
		u_2 \\
  		0
	\end{matrix}\right]	+$$
	$$
+b_{101}
	\left[\begin{matrix}
		0 \\
		zu_2 \\
  		0
	\end{matrix}\right]	
+b_{201}
	\left[\begin{matrix}
		3zu_1u_2 \\
		z^2u_2 \\
  		0
	\end{matrix}\right]	
+c_{010}
	\left[\begin{matrix}
		0 \\
		0 \\
  		u_1
	\end{matrix}\right] 
+c_{110}
	\left[\begin{matrix}
		0 \\
		0 \\
  		zu_1
	\end{matrix}\right] 
	+c_{210}
	\left[\begin{matrix}
		-zu_1u_2 \\
		0 \\
  		z^2u_1
	\end{matrix}\right]
+c_{011}
	\left[\begin{matrix}
		0 \\
		0 \\
  		u_1u_2
	\end{matrix}\right].
$$


We now need the module structure of $M = H^0(W_3,TW_3)$ as a module over global functions. 
So, we first write the generators and relations among them. 
We establish the notation for the generators:
\[
e_1=
	\left[\begin{matrix}
		u_1 \\
		0 \\
  		0
	\end{matrix}\right],
e_2=	
	\left[\begin{matrix}
		u_1u_2 \\
		0 \\
  		0
	\end{matrix}\right],
e_3=
	\left[\begin{matrix}
		0 \\
		1 \\
  		0
	\end{matrix}\right],
e_4=
	\left[\begin{matrix}
		0 \\
		z \\
  		0
	\end{matrix}\right],
e_5=
	\left[\begin{matrix}
		0 \\
		z^2 \\
  		0
	\end{matrix}\right]	,
e_6=
	\left[\begin{matrix}
		3z ^2u_1 \\
		z^3 \\
  		0
	\end{matrix}\right] ,
e_7=
	\left[\begin{matrix}
		0 \\
		0 \\
  		u_1
	\end{matrix}\right] ,
\]
\[	
e_8=
	\left[\begin{matrix}
		0 \\
		0 \\
  		zu_1
	\end{matrix}\right],
e_9=
	\left[\begin{matrix}
		-zu_1u_2 \\
		0 \\
  		z^2u_1
	\end{matrix}\right],	
e_{10}=
	\left[\begin{matrix}
		0 \\
		u_2 \\
  		0
	\end{matrix}\right],
e_{11}=
	\left[\begin{matrix}
		0 \\
		zu_2 \\
  		0
	\end{matrix}\right],
e_{12}=
	\left[\begin{matrix}
		3zu_1u_2 \\
		z^2u_2 \\
  		0
	\end{matrix}\right],
e_{13}=
	\left[\begin{matrix}
		0 \\
		0 \\
  		u_1u_2
	\end{matrix}\right].
\]

These then satisfy the equations:
       $$\begin{array}{l}
  e_2  - u_2 e_1=0 \\
  e_{10} - u_2 e_3 =0\\
e_{13} - u_2e_7=0 \\
\end{array}
\quad\quad \begin{array}{l}
ze_{12} - u_2 e_6=0 \\
ze_{13} - u_2e_8=0  \\
e_{11} - z e_{10}=0 \\
\end{array}
\quad \begin{array}{l}
 e_4 - z e_3 =0  \\
   e_5 - ze_4 =0 \\
e_8 - ze_7=0  \\
\end{array} \quad 
\begin{array}{l}
e_6 - z e_5 - 3z^2e_1=0 \\
e_9 - ze_8 + ze_2=0 \\
e_{12} - ze_{11} - 3ze_1=0 .\\
\end{array}$$  

Since neither $z$ nor $u_2$ are global functions, 
we multiply the equations by $u_1$ to obtain relations over 
$\mathcal O(W_3)$, obtaining the claimed module structure.
\end{proof}

We  now proceed  to investigate the question of isomorphism of Poisson structures.

\begin{lemma}\label{isos3}
There are isomorphisms
$e_3 \simeq e_6$,
$e_7 \simeq e_9$, and
$e_{10} \simeq e_{12}$.
\end{lemma}

\begin{proof}
For each isomorphism use the transition function of $\Lambda^2TW_3$ and then exchange the $U$ and $V$ charts
as in the proof of \Cref{iso1-5}. 
\end{proof}

There are  natural inclusions of the surfaces $\mathbb C$, $Z_{-1}$, and $Z_3$ into  $W_3$: 

\begin{notation}
We denote by $j_s $ for $s=0,1,2$ the inclusions of $\mathbb C^2$, $Z_{-1}$, and $Z_3$ into the threefold $W_3$.
Hence, in coordinates we have: 
\begin{itemize}
\item   $j_0\colon \mathbb C^2 \rightarrow W_3$ includes $\mathbb C^2$ as the fiber  $z=0$,  
taking $\frac{\partial}{\partial u} \wedge \frac{\partial}{\partial v} \mapsto \partial_1 \wedge \partial_2$
\item $j_1\colon Z_3 \rightarrow W_3$ includes $Z_3$ as  $u_2=0$, taking $\frac{\partial}{\partial z} \wedge \frac{\partial}{\partial u} \mapsto \partial_0 \wedge \partial_1$
\item $j_2\colon Z_{-1} \rightarrow W_3$ includes $Z_{-1}$ as $u_1=0$, taking $\frac{\partial}{\partial z} \wedge \frac{\partial}{\partial u} \mapsto \partial_0 \wedge \partial_2$.
\end{itemize}
\end{notation}

\begin{theorem} \label{emb3}
The embeddings of  Poisson surfaces $j_0(\mathbb C^2, \pi_0),$ $j_1(Z_3,\pi_i)$ with $i=0,1,2$ and $j_2(Z_{-1},\pi_i)$ 
with $i=0,1,2,3$
 generate 
all Poisson structures on $W_3$.
\end{theorem}

\begin{proof} 

Let $j_1(Z_3)$ (resp. $j_2(Z_{-1})$) be the embedding of the surface 
$Z_3$ (resp. $Z_{-1}$) into  $W_3$ by $u_2=v_2=0$ (resp. $u_1=v_1=0$). 
 Then Poisson structures induced by the first embedding are:
 $$(j_1)_*(1)_U= \left[\begin{matrix}0 \\ 0 \\ 1\end{matrix}\right]_U,\quad 
(j_1)_*(-\xi^2 v)_V= \left[\begin{matrix}0 \\ 0 \\ -\xi^2 v\end{matrix}\right]_V, \quad \text{hence} \quad 
e_7\vert_{j_1(Z_3)}=(j_1)_*\pi_0. $$
Analogously, $e_8\vert_{j_1(Z_3)}=j_1(\pi_0)$.
%
The Poisson structures induced by the second embedding are 
$$(j_2)_*(1)_U = \left[\begin{matrix}0 \\ 1 \\ 0\end{matrix}\right]_U, \quad 
(j_2)_*(-\xi^3)_V = \left[\begin{matrix}0 \\ -\xi^3\\ 0\end{matrix}\right]_V, \quad \text{hence} \quad 
 e_3\vert_{j_2(Z_{-1})}=(j_2)_*\pi_0.$$
Analogously
$e_4\vert_{j_2(Z_{-1})}=(j_2)_*\pi_1$, 
$e_5\vert_{j_2(Z_{-1})}=(j_2)_*\pi_2$.

Next, take $g \colon Z_{-1} \to \mathbb{C}$  defined by
$g\vert_U (z,u) = u$ and $g\vert_V (\xi,v) = \xi^{-1}v$,
then 
$e_{10}\vert_{j_2(Z_{-1})} =( j_2)_* (g .(1,-\xi^3))$ 
and
$e_{11}\vert_{j_2(Z_{-1})} =( j_2)_*(g. (z,-\xi^2 ))$.
Finally, 
$(j_0)_*(u) = e_1= u_1 \partial_1 \wedge \partial_2$ and
$(j_0)_*(uv) = e_2= u_1u_2 \partial_1 \wedge \partial_2$  give  $e_1$ and $e_2$. 
\end{proof}


\subsection{Poisson cohomology for \texorpdfstring{$\bm{W_3}$}{W\_3}}\label{cohw3}

\begin{lemma}\label{3classes}
The generators of Poisson structures on $W_3$ are divided up into 3 groups
according to their Casimir functions:
$\{e_1, e_2\}, \quad \{e_3, e_4,e_5,e_{10},e_{11}\}, \quad \{e_7, e_8,e_{13}\}.$
\end{lemma}

\begin{proof}
By \Cref{W3gens} all isomorphism classes of Poisson structures on $W_3$ 
are generated by $
e_1 ,
e_2 ,
e_3 ,
e_4 ,
e_5,
e_7,
e_8,
e_{10} ,
e_{11} ,
e_{13}$ and 
calculating $0$-th Poisson cohomology, we obtain:

$$\mbox{Cas}(\pi) =\left\{\begin{array}{lll}
 f \in \mathcal{O}(W_3) / f = f(z)  & if &   \pi = e_1, e_2,\\
 f \in \mathcal{O}(W_3) / f = f(u_1)& if & \pi=
e_3, e_4,e_5,e_{10},e_{11},\\
f \in \mathcal{O}(W_3) / f = f(u_2) & if &  \pi = e_7, e_8,e_{13}.
\end{array}
\right.$$
\end{proof}

\begin{lemma}\label{alphas}
The Poisson manifolds $(W_3,e_1)$ and $(W_3,e_2)$ are not isomorphic.
\end{lemma}

\begin{proof} We have that $e_1 = u_1 \partial u_1\wedge \partial u_2$ and $e_2= u_2 e_1$, 
or written as section of $\Lambda^2TW_3$ we have
$$
e_1 = 
	\left[\begin{matrix}
		u_1 \\
		0 \\
  		0
	\end{matrix}\right]_U
	=  \left[\begin{matrix}
		\xi v_1 \\
		0 \\
  		0
	\end{matrix}\right]_V , \quad 
e_2 = 	
	\left[\begin{matrix}
		u_1u_2 \\
		0 \\
  		0
	\end{matrix}\right]_U= 
	\left[\begin{matrix}
		 v_1v_2 \\
		0 \\
  		0
	\end{matrix}\right]_V
	$$

Therefore the degeneracy loci  have the following irreducible components:
$$D(e_1) = \{u_1=0\} \cup\{\xi=0\} \cup \{v_1=0\},$$
$$D(e_2) = \{u_1=0\} \cup\{u_2=0\} \cup \{v_1=0\} \cup\{v_2=0\}.$$
These have different number of irreducible components, implying $\alpha_0$ is not 
isomorphic to $\alpha_1$.
\end{proof}

\begin{lemma}\label{betas}
The Poisson manifolds $(W_3,e_3)$ , $(W_3,e_4)$, $(W_3,e_5)$, $(W_3,e_{10})$ and $(W_3,e_{11})$ 
are  pairwise nonisomorphic.
\end{lemma}

\begin{proof}
We compute the degeneracy loci of the Poisson structures.
We have that 
$$
e_3 = 
	\left[\begin{matrix}
		0 \\
		 1 \\
  		0
	\end{matrix}\right]_U
	=  \left[\begin{matrix}
		-3\xi^2v_1 \\
		 -\xi^3\\
  		0
	\end{matrix}\right]_V , \quad 
e_4 = 	
	\left[\begin{matrix}
		0 \\
		z \\
  		0
	\end{matrix}\right]_U= 
	\left[\begin{matrix}
		-3\xi v_1 \\
		-\xi^2 \\
  		0
	\end{matrix}\right]_V, 
e_5 = 
	\left[\begin{matrix}
		0 \\
		z^2 \\
  		0
	\end{matrix}\right]_U
	=
	\left[\begin{matrix}
		-3v_1 \\
		-\xi \\
  		0
	\end{matrix}\right]_V
	$$
$$
e_{10} = 
	\left[\begin{matrix}
		0 \\
		u_2 \\
  		0
	\end{matrix}\right]_U
	=  \left[\begin{matrix}
		-3\xi v_1v_2 \\
		-\xi^2 v_2 \\
  		0
	\end{matrix}\right]_V , \quad 
e_{11} = 	
	\left[\begin{matrix}
		0 \\
		zu_2 \\
  		0
	\end{matrix}\right]_U= 
	\left[\begin{matrix}
		 -3v_1v_2 \\
		-\xi v_2\\
  		0
	\end{matrix}\right]_V.
	$$

The degeneracy loci are:
\begin{align*}
D(e_3) & = \{ \xi=0 \}, \\
D(e_4) & = \{ z=0 \} \cup \{ \xi=0 \}, \\
D(e_5) & = \{ z=0 \} \cup \{ \xi=v_1=0 \}, \\
D(e_{10}) & = \{u_2=0\} \cup \{\xi=0\} \cup \{v_2=0\}, \\
D(e_{11}) & = \{z=0\} \cup \{u_2=0\} \cup \{v_2=0\} \cup \{\xi=v_1=0\}, 
\end{align*}
these  are pairwise nonisomorphic, and thus also
their corresponding Poisson structures.
\end{proof}

\begin{lemma}\label{gammas}
The Poisson manifolds $(W_3,e_7)$, $(W_3,e_8)$ and $(W_3,e_{13})$ are pairwise nonisomorphic.
\end{lemma}

\begin{proof}
We compute the degeneracy loci of the Poisson structures.
We have that 
$$
e_7 = 
	\left[\begin{matrix}
		0 \\
		 0 \\
  		u_1
	\end{matrix}\right]_U
	=  \left[\begin{matrix}
		\xi v_1 v_2\\
		 0\\
  		-\xi^2v_1
	\end{matrix}\right]_V,
e_8 = 	
	\left[\begin{matrix}
		0 \\
		0\\
  		zu_1
	\end{matrix}\right]_U= 
	\left[\begin{matrix}
		 v_1v_2 \\
		0\\
  		-\xi v_1
	\end{matrix}\right]_V, 
e_{13} = 
	\left[\begin{matrix}
		0 \\
		0 \\
  		u_1u_2
	\end{matrix}\right]_U
	=
	\left[\begin{matrix}
		v_1v_2^2 \\
		0 \\
  	-	\xi v_1v_2
	\end{matrix}\right]_V.
	$$

So, the degeneracy loci:
\begin{align*}
D(e_7) & = \{u_1 = 0\} \cup \{ \xi=0 \} \cup \{ v_1=0 \}, \\
D(e_8) & =\{z = 0\} \cup \{ u_1 = 0 \} \cup \{ \xi = v_2 = 0\} \cup \{v_1 = 0\}, \\
D(e_{13}) & = \{ u_1 = 0 \} \cup \{ u_2 = 0 \} \cup \{ v_1 = 0 \} \cup \{ v_2 = 0 \},
\end{align*}
are pairwise nonisomorphic, and therefore the corresponding Poisson 
structures are distinct.
\end{proof}

This  concludes the  lemmata needed to prove \Cref{W3gens},
showing that all 10 listed isomorphism classes of Poisson structures are indeed distinct.

\subsection{Symplectic foliations on \texorpdfstring{$\bm{W_3}$}{W\_3}}

We have seen that all possible Poisson structures on $W_3$ can be obtained from 
$e_1 ,
e_2 ,
e_3 ,
e_4 ,
e_5,
e_7,
e_8,
e_{10} ,
e_{11} ,
e_{13}$   given in \Cref{W3gens}. Their corresponding symplectic foliations have
$0$-dimensional  leaves consisting of single points inside their degeneracy loci described in \Cref{3classes},
 and outside these loci, all symplectic leaves 
are $2$-dimensional and can be described as follows.

\begin{theorem}\label{W3foliation}
The symplectic foliations on $W_3$ 
have 0-dimensional  leaves consisting of single points over each 
of their corresponding  degeneracy loci described in
the proofs of Lemmata \ref{alphas}, \ref{betas}, \ref{gammas}, and their generic  leaves, which  are 2-dimensional, 
are as follows:

\begin{itemize}
\item  Isomorphic to $\mathbb C^*\times \mathbb C$ for $e_1$.
\item  Isomorphic to $\mathbb C^*\times \mathbb C^*$ for $e_2$.
\item  Surfaces of constant $u_1$, for $e_3, e_4,e_5, e_{10}, e_{11}$ and $e_{13}$. 
\item Surfaces of constant $u_2$,  for $e_7$ and $e_8$.
\end{itemize}

\end{theorem}

\begin{proof}
Combine the Casimir functions given in the proof of \Cref{3classes} with \Cref{leaves}.
\end{proof}

\paragraph{\bf Acknowledgements}  
We are grateful to Brent Pym for kindly explaining to us some of the fundamental notions of Poisson geometry.
E. Ballico is a member of  GNSAGA of INdAM (Italy).
E. Gasparim thanks the Department of Mathematics of  the University of Trento for the 
support and  hospitality.
B. Suzuki was supported by the ANID-FAPESP cooperation 2019/13204-0.

\end{document}